\documentclass{amsart}
\usepackage{arydshln}
\usepackage{mathabx}
\usepackage{mathtools}
\usepackage{amssymb,
amsmath,
amsthm,
graphicx,
amscd,
multind,
eufrak,
hyperref,
}
\theoremstyle{plain}
\newtheorem{thm}{Theorem}[section]
\newtheorem{lm}[thm]{Lemma}
\newtheorem{prop}[thm]{Proposition}
\newtheorem{conj}[thm]{Conjecture}

\theoremstyle{definition}
\newtheorem{de}[thm]{Definition}
\newtheorem{ex}[thm]{Example}

\newcommand{\CC}{{\mathbb C}}
\newcommand{\RR}{{\mathbb R}}
\newcommand{\ZZ}{{\mathbb Z}}
\newcommand{\PP}{{\mathbb P}}
\newcommand{\im}{\operatorname{im}}
{\begin{figure} \begin{center}}%
{\end{center} \end{figure}}

\newcommand{\Sym}{\operatorname{Sym}}
\newcommand{\diag}{\operatorname{diag}\nolimits}

\newcommand{\be}{\mathbf{e}}
\newcommand{\bu}{\mathbf{u}}
\newcommand{\bv}{\mathbf{v}}

\newcommand{\OO}{\operatorname{O}\nolimits}
\newcommand{\SO}{\operatorname{SO}\nolimits}

\newcommand{\dd}{\mathrm{d}}
\newcommand{\lieg}[1]{\mathrm{#1}}

\newcommand{\Crit}{\mathrm{Crit}}
\renewcommand{\phi}{\varphi}
\newcommand{\R}{\mathbb{R}}
\newcommand{\comment}[1]{}

\begin{document}

\begin{abstract}
Motivated by the many potential applications of low-rank multi-way tensor
approximations, we set out to count the rank-one tensors that are critical
points of the distance function to a general tensor $v$. As this count
depends on $v$, we average over $v$ drawn from a Gaussian distribution,
and find a formula that relates this average to problems in random
matrix theory.
\end{abstract}

\title[Critical rank-one approximations]{The average number
of critical\\ rank-one approximations to a tensor}
\author{Jan Draisma}
\address{
Department of Mathematics and Computer Science\\
Technische Universiteit Eindhoven\\
P.O. Box 513, 5600 MB Eindhoven, The Netherlands\\
and Vrije Universiteit Amsterdam and Centrum Wiskunde en
Informatica, Amsterdam, The Netherlands.}
\thanks{JD is supported by a Vidi grant from the Netherlands
Organisation for Scientific Research (NWO)}
\email{j.draisma@tue.nl}
\author{Emil Horobe\c{t}}
\address{
Department of Mathematics and Computer Science\\
Technische Universiteit Eindhoven\\
P.O. Box 513, 5600 MB Eindhoven, The Netherlands.}
\thanks{EH is supported by an NWO free competition grant}
\email{e.horobet@tue.nl}
\maketitle

\section{Introduction} \label{sec:Intro}
Low-rank approximation of {\em matrices} via singular value decomposition
is among the most important algebraic tools for solving approximation
problems in data compression, signal processing, computer vision,
 etc. Low-rank approximation for {\em tensors} has the
same application potential, but raises substantial mathematical and
computational challenges. To begin with, tensor rank and many related
problems are NP-hard \cite{Hastad90,Hillar10}, although in low degrees
(symmetric) tensor decomposition has been approached computationally in
\cite{Brachat10,Oeding13} by greatly generalising classical techniques
due to Sylvester and contemporaries. Furthermore, tensors of bounded
rank do not form a closed subset, so that a best low-rank approximation
of a tensor on the boundary does not exist \cite{deSilva08}. This
latter problem does not occur for tensors of rank at most one, which
{\em do} form a closed set, and where the best rank-one approximation
does exist under a suitable genericity assumption \cite{Friedland12b}.

In spite of these mathematical difficulties, much application-oriented
research revolves around algorithms for computing low-rank approximations
\cite{Belzen08,Belzen09,Comon08,DeLathauwer08a,DeLathauwer08b,DeLathauwer08c,
Lim05,Ishteva11}. Typically, these algorithms are of a local nature and would
get into problems near non-minimal critical points of the distance
function to be minimised. This motivates our study into the question of
{\em how many} critical points one should expect in the easiest nontrivial
setting, namely that of approximation by rank-one tensors.  This number
should be thought of as a measure of the complexity of finding the closest
rank-one approximation. The corresponding complex count, which is the
topic of \cite{Friedland12b} and with which we will compare our results,
measures the degree of an algebraic field extension needed to write down
the critical points as algebraic functions of the tensor to be minimised.
We will treat both ordinary tensors and symmetric tensors.

\subsection*{Ordinary tensors}
To formulate our problem and results, let $n_1,\ldots,n_p$ be natural
numbers and let $X \subset V:=\R^{n_1}\otimes\cdots \otimes \R^{n_p}$
be the variety of rank-one $p$-way tensors, i.e., those that can be
expressed as $x_1 \otimes x_2 \otimes \cdots \otimes x_p$ for vectors
$x_i \in \RR^{n_i},\ i=1,\ldots,p$. Given a general tensor $v \in
V:=\R^{n_1}\otimes\cdots \otimes \R^{n_p}$, one would like to compute
$x \in X$ that minimizes the squared Euclidean distance \[d_v(x) =
\sum_{i_1,\ldots,i_p} (v_{i_1,\ldots,i_p}-x_{i_1,\ldots,i_p})^2\]
from $v$. For the matrix case, where $p=2$, this minimizer is
$\sigma x_1 x_2^T$ where $\sigma$ is the largest singular value of
$v$ and $x_1,x_2$ are the corresponding left and right singular
vectors. Indeed, all critical points of $d_v$ are of this form,
with $\sigma$ running through all singular values of $v$. For $p>2$,
several algorithms have been proposed for rank-one approximation (see,
e.g., \cite{Belzen07,DeLathauwer00}). These algorithms have a local nature
and experience difficulties near critical points of $d_v$. This is one
of our motivations for {\em counting} these critical points---the main
goal of this paper.

In \cite{Friedland12b}, a general formula is found for the number of {\em
complex} critical points of $d_v$ on $X_\CC$. In this case the $x_i$ can
have complex coefficients and the expression $d_v$ is copied verbatim,
i.e., without inserting complex conjugates. This means that $d_v(x)$
does not really measure a (squared) distance---e.g., it can be zero even
for $x \neq v$---but on the positive side the number of critical points
of $d_v$ on $X_\CC$ is constant for $v$ away from some hypersurface
(which in particular has measure zero) and this constant is the top
Chern class of some very explicit vector bundle \cite{Friedland12b}. 
For more information on this hypersurface, see \cite[Section 7]{Draisma13c} and
\cite{Horobet15a}. Explicit equations for these hypersurfaces are not
known, even in our setting. 

Over the {\em real} numbers, which we consider, the number of critical
points of $d_v$ can jump as $v$ passes through (the real locus of) the
same hypersurface. Typically, it jumps by $2$, as two real cricital points
come together and continue as a complex-conjugate pair of critical points.
To arrive at a single number, we therefore impose
a probability distribution on our data space $V$ with density function
$\omega$ (soon specialized to a standard multivariate Gaussian), and we
ask: what is the {\em expected number} of critical points of $d_v$ when
$v$ is drawn from the given probability distribution? In other words,
we want to compute
\[
\int\limits_{\R^{n_1}\otimes \cdots \otimes \R^{n_p}} \#\{\text{real critical points of } d_v \text{ on } X \} \omega(v) dv.
\]
This formula is complicated for two different reasons. First, given
a point $v \in V$, the value of the integrand at $v$ is not easy to
compute. Second, the integral is over a space of dimension $N:=\prod_i
n_i$, which is rather large even for small values of the $n_i$. The main
result of this paper is the following formula for the above integral, in
the Gaussian case, in terms of an integral over a space of much smaller
dimension quadratic in the number $n:=\sum_i n_i$.

\begin{thm} \label{thm:Main}
Suppose that $v \in V$ is drawn from the (standard) multivariate Gaussian
distribution with (mean zero and) density function
\[ \omega(v):=\frac{1}{(2\pi)^{N/2}} e^{-(\sum_{\alpha} v_\alpha^2)/2}, \]
where the multi-index $\alpha$ runs over $\{1,\ldots,n_1\} \times \cdots \times
\{1,\ldots,n_p\}$. Then the expected number of critical points of $d_v$
on $X$ equals
\[
\frac{(2\pi)^{p/2}}{2^{n/2}}\frac{1}{\prod_{i=1}^p\Gamma\left(\frac{n_i}{2}\right)}\int\limits_{W_1}\left|\det
C(w_1)\right| \dd \mu_{W_1}.
\]
Here $W_1$ is a space of dimension $1+\sum_{i<j} (n_i-1)(n_j-1)$
with coordinates $w_0 \in \RR$ and $C_{ij} \in \RR^{(n_i-1) \times
(n_j-1)}$ with $i<j$, $C(w_1)$ is the symmetric $(n-p) \times
(n-p)$-matrix of block shape
\[
\begin{bmatrix}
w_0I_{n_1-1} & C_{1,2} & \cdots & C_{1,p} \\
C_{1,2}^T & w_0I_{n_2-1} & \cdots & C_{2,p}\\
\vdots & \vdots & & \vdots \\
C_{1,p}^T & C_{2,p}^T & \cdots & w_0 I_{n_p-1}
\end{bmatrix},
\]
and $\mu_{W_1}$ makes $w_0$ and the $\sum_{i<j} (n_i-1) \cdot
(n_j-1)$ matrix entries of the $C_{i,j}$ into independent, standard
normally distributed variables. Moreover, $\Gamma$ is Euler's gamma
function.
\end{thm}

Not only the dimension of the integral has dropped considerably, but also
the integrand can be evaluated easily. The following example illustrates
the case where all $n_i$ are equal to $2$.

\begin{ex} \label{ex:222}
Suppose that all $n_i$ are equal to $2$. Then the matrix
$C(w_1)$
becomes
\[ C(w_1)=
\begin{bmatrix}
w_0 & w_{12} & \cdots & w_{1p} \\
w_{12} & w_0 & \cdots & w_{2p} \\
\vdots & \vdots & & \vdots \\
w_{1p} & w_{2p} & \cdots & w_0
\end{bmatrix}
\]
where the distinct entries are independent scalar variables $\sim
\mathcal{N}(0,1)$. The expected number of critical points of $d_v$ on
$X$ equals
\[ \frac{(2\pi)^{p/2}}{2^{2p/2}} \frac{1}{\Gamma(\frac{1}{1})^p}
\mathbb{E}(|\det(C(w_1))|)
=
\left(\frac{\pi}{2}\right)^{p/2} \mathbb{E}(|\det(C(w_1))|),
\]
where the latter factor is the expected absolute value of the determinant
of $C(w_1)$. For $p=2$ that expected value of $|w_0^2-w_{12}^2|$ can
be computed symbolically and equals $4/\pi$. Thus the expression above
then reduces to $2$, which is just the number of singular values of a
$2 \times 2$-matrix. For higher $p$ we do not know a closed form
expression for $\mathbb{E}(|\det(C(w_1))|),$ but we will present some
numerical approximations in Section~\ref{sec:Values}.
\hfill $\diamondsuit$
\end{ex}

In Section~\ref{sec:Tensors} we prove Theorem~\ref{thm:Main}, and
in Section~\ref{sec:Values} we list some numerically computed values.
These values lead to the following intriguing {\em stabilization
conjecture}.

\begin{conj}
Suppose that $n_p-1>\sum_{i=1}^{p-1} n_i-1$. Then, in the Gaussian setting
of Theorem~\ref{thm:Main}, the expected number of critical points of $d_v$
on $X$ does not decrease if we replace $n_p$ by $n_p-1$.
\end{conj}

For $p=2$ this follows from the statement that the number of singular
values of a sufficiently general $n_1 \times n_2$-matrix with $n_1<n_2$
equals $n_1$, which in fact remains the {\em same} when
replacing $n_2$ by $n_2-1$. For arbitrary $p$ the statement is true over
$\CC$ as shown in \cite{Friedland12b}, again with equality, but the proof is not bijective.
Instead, it uses vector bundles and Chern
classes, techniques that do not carry over to our setting. It would be
very interesting to find a direct geometric argument that {\em does}
explain our experimental findings over the reals, as well.

\begin{ex}
Alternatively, one could try and prove the conjecture directly from the
integral formula in Theorem~\ref{thm:Main}. The smallest open case is when $p=3$ and $(n_1,n_2,n_3)=(2,2,4)$, and here the conjecture says that
\begin{align*}
\frac{\sqrt{\pi}}{2\sqrt{2}}&\int\limits_{\mathbb{R}}\int\limits_{\mathbb{R}^7}\left|\det\left(
                                            \begin{array}{c:c:ccc}
                                              w_0 & w_{12} & w_{13} & w_{14} & w_{15} \\\hdashline
                                              w_{12} & w_0 & w_{23} & w_{23} & w_{25}\\ \hdashline
                                              w_{13} & w_{23} & w_0 & 0 &0\\
                                              w_{14} & w_{24} & 0 & w_0 &0\\
                                              w_{15} & w_{25} & 0 & 0 & w_0
                                            \end{array}
                                          \right)
\right|e^{-\frac{w_0^2+\sum w_{ij}^2}{2}}\mathrm{d}w_0 \mathrm{d}w_{ij}\\
\\
\leq&\int\limits_{\mathbb{R}}\int\limits_{\mathbb{R}^5}\left|\det\left(
                                            \begin{array}{c:c:cc}
                                              w_0 & w_{12} & w_{13} & w_{14} \\\hdashline
                                              w_{12} & w_0 & w_{23} & w_{24}\\ \hdashline
                                              w_{13} & w_{23} & w_0 & 0 \\
                                              w_{14} & w_{24} & 0 & w_0
                                            \end{array}
                                          \right)
\right|e^{-\frac{w_0^2+\sum w_{ij}^2}{2}}\mathrm{d}w_0 \mathrm{d}w_{ij}.
\end{align*}
The determinant in the first integral is approximately $w_0$ times a
determinant like in the second integral, but we do not know how to turn
this observation into a proof of this integral inequality.
\hfill $\diamondsuit$
\end{ex}

\subsection*{Symmetric tensors}
In the second part of this paper, we discuss {\em symmetric
tensors}. There we consider the space $V=S^p \RR^n$ of homogeneous
polynomials of degree $p$ in the standard basis $e_1,\ldots,e_n$ of
$\RR^n$, and $X$ is the subvariety of $V$ consisting of all polynomials
that are of the form $\pm u^p$ with $u \in \RR^n$. We equip $V$
with the {\em Bombieri norm}, in which the monomials in the $e_i$
form an orthogonal basis with squared norms \[ ||e_1^{\alpha_1} \cdots
e_n^{\alpha_n}||^2=\frac{\alpha_1! \cdots \alpha_n!}{p!}. \] Our result
on the average number of critical points of $d_v$ on $X$ is as follows.

\begin{thm} \label{thm:Symm}
When $v \in S^p \RR^n$ is drawn from the standard Gaussian
distribution relative to the Bombieri norm, then the
expected number of critical points of $d_v$ on the variety
of (plus or minus) pure $p$-th powers equals
\begin{align*}
&\frac{1}{2^{(n^2+3n-2)/4} \prod_{i=1}^n
\Gamma(i/2)}
\int\limits_{\lambda_2\leq \ldots\leq
\lambda_{n}}\int\limits_{-\infty}^{\infty}\left(\prod_{i=2}^{n}|\sqrt{p}w_0-
\sqrt{p-1} \lambda_i|\right)\\
&\cdot
\left(\prod_{i<j}(\lambda_j-\lambda_i)
\right) e^{-w_0^2/2-\sum_{i=2}^n \lambda_i^2/4}
\dd
w_0 \dd \lambda_2 \cdots \dd \lambda_n.
\end{align*}
\end{thm}

Here the dimension reduction is even more dramatic: from an integral over
a space of dimension $\binom{p+n-1}{p}$ to an integral over a polyhedral
cone of dimension $n$. In this case, the corresponding complex count
is already known from \cite{Cartwright13}: it is the geometric series
$1+(p-1)+\cdots+(p-1)^{n-1}$.

\begin{ex} \label{ex:Symm}
For $p=2$ the integral above evaluates to $n$ (see
Subsection~\ref{ssec:Symmats} for a direct computation).  Indeed, for $p=2$
the symmetric tensor $v$ is a symmetric matrix, and the critical points
of $d_v$ on the manifold of rank-one symmetric matrices are those of
the form $\lambda u u^T$, with $u$ a norm-$1$ eigenvector of $v$
with eigenvalue $\lambda$.

For $n=2$ it turns out that the above integral can also be evaluated in
closed form, with value $\sqrt{3p-2}$; a different proof of this fact
appeared in \cite{Draisma13c}. For $n=3$ we provide a closed formula
in Section~\ref{sec:Values}. In all of these cases, the average count is
an algebraic number. We do not know if this persists for larger values
of $n$. \hfill $\diamondsuit$
\end{ex}

\subsection*{Outline}
The remainder of this paper is organized as follows. First, in
Section~\ref{sec:DoubleCounting} we explain a double counting strategy
for computing the quantity of interest. This strategy is then applied to
ordinary tensors in Section~\ref{sec:Tensors} and to symmetric tensors
in Section~\ref{sec:SymTensors}.  We conclude with some (symbolically
or numerically) computed values in Section~\ref{sec:Values}.

\subsection*{Acknowledgments}
This paper fits in the research programme laid out in \cite{Draisma13c},
which asks for {\em Euclidean distance degrees} of algebraic varieties
arising in applications. We thank the authors of that paper, as well as
our Eindhoven colleague Rob Eggermont, for several stimulating discussions
on the topic of this paper.

\section{Double counting} \label{sec:DoubleCounting}

Suppose that we have equipped $V=\RR^N$ with an inner product $(.|.)$
and that we have a smooth manifold $X \subseteq V$. Assume that we
have a probability density $\omega$ on $V=\RR^N$ and that we want to
count the average number of critical points $x \in X$ of the function
$d_v(x):=(v-x|v-x)$ when $v$ is drawn according to that density. Let $\Crit$
denote the set
\[ \Crit:=\{(v,x) \mid v-x \perp T_x X\} \subseteq V \times X \]
of pairs $(v,x) \in X \times V$ for which $x$ is a critical point of
$d_v$. For fixed $x \in X$ the $v \in V$ with
$(v,x) \in \Crit$ form an affine space, namely, $x+(T_x X)^\perp$. In
particular, $\Crit$ is a manifold of dimension $N$. On the other hand,
for fixed $v \in V$, the $x \in X$ for which $(v,x) \in \Crit$ are what
we want to count. Let $\pi_V:\Crit \to V$ be the first projection. Then
(the absolute value of) the pull-back $|\pi_V^* \omega \dd v|$ is a pseudo
volume form on $\Crit$, and we have
\[
\int_{V} \#(\pi_V^{-1}(v)) \omega(v) \dd v=
\int_{\Crit} 1 |\pi_V^* \omega \dd v|.
\]
Now suppose that we have a smooth $1:1$ parameterization $\phi:\RR^N
\to \Crit$ (perhaps defined outside some set of measure zero). Then the
latter integral is just
\[ \int_{\RR^N} |\det J_w (\pi_V \circ \phi)| \omega(\pi_V(\phi(w))) \dd w, \]
where $J_w (\pi_V \circ \phi)$ is the Jacobian of $\pi_V \circ \phi$ at
the point $w$. We will see that if $X$ is the manifold of rank-one tensors
or rank-one symmetric tensors, then $\Crit$ (or in fact, a slight variant
of it) has a particularly friendly parameterization, and we will use the
latter expression to compute the expected number of critical points of
$d_v$. In a more general setting, this double-counting approach is
discussed in \cite{Draisma13c}.

\section{Ordinary tensors} \label{sec:Tensors}

\subsection{Set-up}
Let $V_1,\ldots,V_p$ be real vector spaces of dimensions $n_1 \leq \ldots
\leq n_p$ equipped with positive definite inner products $(.|.)$. Equip
$V:=\bigotimes_{i=1}^p V_i$, a vector space of dimension $N:=n_1 \cdots
n_p$, with the induced inner product and associated norm, also denoted
$(.|.)$. Given a tensor $v \in V$, we want to count the number of critical
points of the function
\[ d_v: x \mapsto ||v-x||^2=(v|v)-2(v|x)+(x|x) \]
on the manifold $X \subseteq V$ of non-zero rank-one tensors $x=x_1
\otimes \cdots \otimes x_p$. The following well-known lemma (see
for instance \cite{Friedland12b}) characterizes which $x$ are critical for a
given $v \in V$. In its statement we extend the notation $(v | u)$
to the setting where $u$ is a tensor in $\bigotimes_{i \in I} V_i$
for some subset $I \subseteq \{1,\ldots,p\}$, to stand for the tensor
in $\bigotimes_{i \not\in I} V_i$ obtained by contracting $v$ with $u$
using the inner products.

\begin{lm}
The non-zero rank-one tensor $x=x_1 \otimes \cdots \otimes x_p$ is
a critical point of $d_v$ if and only if for all $i=1,\ldots,p$
we have
\[ (v | x_1 \otimes \cdots \otimes \hat{x_i} \otimes \cdots \otimes x_p)=
\left(\prod_{j \neq i}(x_j|x_j) \right) x_i. \]
\end{lm}

In words: pairing $v$ with the tensor product of the $x_j$ with $j
\neq i$ gives a well-defined scalar multiple of $x_i$, and this should
hold for all $i$.

\begin{proof}
The tangent space at $x$ to the manifold of rank-one tensors is
$\sum_{i=1}^p x_1 \otimes \cdots \otimes V_i \otimes \cdots \otimes
x_p$. Fixing $i$ and $y \in V_i$, the derivative of $d_v$ in the
direction $x_1 \otimes \cdots \otimes y \otimes \cdots \otimes x_p$
is
\[ -2 (v - x_1 \otimes \cdots \otimes x_p |
x_1 \otimes \cdots \otimes y \otimes \cdots \otimes x_p).
\]
Equating this to zero for all $y$ yields that
\[ (v | x_1 \otimes \cdots \otimes \hat{x_i} \otimes \cdots
\otimes x_p)=(x_1 \otimes \cdots \otimes x_p | x_1 \otimes \cdots
\otimes \hat{x_i} \otimes \cdots \otimes x_p) = \left(\prod_{j \neq i}
(x_j|x_j) \right) x_i, \]
as claimed.
\end{proof}

The lemma can also be read as follows: a rank-one tensor $x_1 \otimes
\cdots \otimes x_p$ is critical for $d_v$ if and only if
first, for each $i$
the contraction $(v|x_1 \otimes \cdots \otimes \hat{x_i} \otimes \cdots
\otimes x_p)$ is some scalar multiple of $x_i$, and second, $(v|x_1
\otimes \cdots \otimes x_p)$ equals $\prod_j (x_j|x_j)$.  From this
description it is clear that if $x_1 \otimes \cdots \otimes x_p$
merely satisfies the first condition, then some scalar multiple of it is
critical for $d_v$. Also, if a rank-one tensor $u$ is critical for $d_v$, then $t
u$ is critical for $d_{t v}$ for all $t \in \RR$.  These considerations
give rise to the following definition and proposition.

\begin{de}
Define $\Crit$ to be the subset of $V \times (\PP V_1 \times \cdots
\times \PP V_p)$ consisting of points $(v,([u_1],\ldots,[u_p]))$
for which all $2 \times 2$-determinants of the $\dim V_i \times 2$-matrix
$[(v|u_1 \otimes \cdots \otimes \hat{u_i} \otimes \cdots \otimes u_p)
\quad \mid \quad u_i ]$ vanish, for each $i=1,\ldots,p$.
\end{de}


\begin{prop}
The projection $\Crit \to \prod_i \PP V_i$ is a smooth sub-bundle of the
trivial bundle $V \times \prod_i \PP V_i$ over $\prod_i \PP V_i$ of rank
$N -(n_1+\cdots+n_p) + p$, while the fiber of the projection
$\pi_V:\Crit \to V$ over a tensor $v$ counts the number of critical points
of $d_v$ in the manifold of non-zero rank-one tensors.
\end{prop}

\begin{proof}
The second statement is clear from the above. For the first observe
that the fiber above $\bu=([u_1],\ldots,[u_p])$ equals $W_\bu \times
\{([u_1],\ldots,[u_p])\}$ where
\[ W_\bu=\left(\bigoplus_{i=1}^p u_1 \otimes \cdots \otimes (u_i)^\bot \otimes
\cdots \otimes u_p\right)^\bot \subseteq V. \]
This space varies smoothly with $\bu$ and has codimension $\sum_i
(n_i-1)$, whence the dimension formula.
\end{proof}

We want to compute the average fiber size of the projection $\Crit \to
V$. Here {\em average} depends on the choice of a measure on $V$, and
we take the Gaussian measure $\frac{1}{(2\pi)^{N/2}} e^{-||v||^2/2}
\dd v$, where $\dd v$ stands for ordinary Lebesgue measure obtained
from identifying $V$ with $\RR^N$ by a linear map that relates $(.|.)$
to the standard inner product on $\RR^N$.

\subsection{Parameterizing $\Crit$} \label{ssec:Parm}

To apply the double counting strategy from
Section~\ref{sec:DoubleCounting}, we
introduce a convenient parameterization of $\Crit$. Fix norm-$1$
vectors $e_i \in V_i,\ i=1,\ldots,p$, write $\be=(e_1,\ldots,e_p)$
and $[\be]:=([e_1],\ldots,[e_p])$, and define
\[ W:=W_{[\be]}=\left(\bigoplus_{i=1}^p e_1 \otimes \cdots \otimes
(e_i)^\bot \otimes \cdots \otimes e_p\right)^\bot. \]
We parameterize (an open subset of) $\PP V_i$ by the map $e_i^\perp
\to \PP V_i, u_i \mapsto [e_i+u_i]$. Write $U:=\prod_{i=1}^p
(e_i^\bot)$. For $\bu=(u_1,\ldots,u_p) \in U$ let $R_\bu$ denote a linear
isomorphism $W \to W_{[\be+\bu]}$, to be chosen later, but at least
smoothly varying with $\bu$ and perhaps defined outside some subvariety
of positive codimension.

Next define
\[ \phi: W \times U \to V,\ (w,\bu) \mapsto R_{\bu} w. \]
Then we have the following fundamental identity
\[
\frac{1}{(2\pi)^{N/2}}
\int_{V} (\# \pi_V^{-1}(v)) \cdot  e^{-\frac{||v||^2}{2}} \dd v\\
=
\frac{1}{(2\pi)^{N/2}}
\int_{W \times U}
	|\det J_{(w,\bu)} \phi| e^{-\frac{||R_\bu w||^2}{2}}
	\dd \bu\ \dd w,
\]
where $J_{(w,\bu)} \phi$ is the Jacobian of $\phi$ at
$(w,\bu)$, whose determinant is measured relative to the volume form on
$V$ coming from the inner product and the volume form on $W \times U$
coming from the inner products of the factors, which are interpreted
perpendicular to each other. The left-hand side is our desired quantity,
and our goal is to show that the right-hand side reduces to the formula in
Theorem~\ref{thm:Main}.

We choose $R_{\bu}$ to be the tensor product $R_{u_1}
\otimes \cdots \otimes R_{u_p}$, where $R_{u_i}$ is the element of
$\lieg{SO}(V_i)$ determined by the conditions that it maps $e_i$ to
a positive scalar multiple of $e_i+u_i$ and that it restricts to the
identity on $\langle e_i,u_i \rangle^\bot$; this map is unique for
non-zero $u_i \in e_i^\bot$. Indeed, we have
\begin{align*}
R_{u_i}
&= \left(I-e_i e_i^T-\frac{u_i}{||u_i||}
\frac{u_i^T}{||u_i||} \right)
+ \left( \frac{e_i + u_i}{\sqrt{1+||u_i||^2}} e_i^T
+ \frac{u_i - ||u_i||^2 e_i}{||u_i|| \sqrt{1+||u_i||^2}}
\frac{u_i^T}{||u_i||} \right)\\
&= \left(I-e_i e_i^T-\frac{u_i u_i^T}{||u_i||^2}
\right)
+ \left( \frac{e_i + u_i}{\sqrt{1+||u_i||^2}} e_i^T
+ \frac{u_i - ||u_i||^2 e_i}{\sqrt{1+||u_i||^2}}
\frac{u_i^T}{||u_i||^2} \right)
\end{align*}
where the first term is the orthogonal projection to $\langle e_i,u_i
\rangle^\bot$ and the second term is projection onto the plane $\langle
e_i,u_i \rangle$ followed by a suitable rotation there.  Two important
remarks concerning symmetries are in order. First, by construction of
$R_{u_i}$ we have
\begin{equation}\label{inv}
R_{u_i}^{-1}=R_{-u_i}.
\end{equation}
Second, for any element $g \in \SO(e_i^\perp) \subseteq
\SO(V_i)$ we have
\begin{equation}\label{inv2}
R_{g u_i}=g \circ R_{u_i} \circ g^{-1}.
\end{equation}
We now compute the
derivative at $u_i$ of the map $e_i^\perp \to \SO(V_i), u \mapsto R_u$ in the direction $v_i
\in e_i^\bot$. First, when $v_i$ is perpendicular to both $e_i$ and $u_i$, this
derivative equals
\begin{equation}\label{derv}
\frac{\partial R_{u_i}}{\partial v_i}=\frac{1}{\sqrt{1+||u_i||^2}} (v_i e_i^T - e_i v_i^T)
- \frac{\sqrt{1+||u_i||^2}-1}{||u_i||^2 \sqrt{1+||u_i||^2}}
(u_i v_i^T + v_i u_i^T).
\end{equation}
Second, when $v_i$ equals $u_i$, the derivative equals
\begin{equation}\label{deru}
\frac{\partial R_{u_i}}{\partial u_i}=\frac{1}{(1+||u_i||^2)^{3/2}} (-u_i u_i^T + u_i e_i^T - e_i
u_i^T - ||u_i||^2 e_i e_i^T ).
\end{equation}
For now, fix $(w,\bu) \in W \times U$. On the subspace $T_w W=W$ of
$T_{(w,\bu)} W \times U$ the Jacobian of $\phi$ is just the map $W \to V,
w \mapsto R_\bu w$. Hence relative to the orthogonal
decompositions $V=W^\perp \oplus W$ and $U \times W$, we have a block
decomposition
\[ R_{\bu}^{-1} J_{(w,\bu)} \phi = \begin{bmatrix} A_{(w,\bu)} & 0 \\ * &
I_{W} \end{bmatrix} \]
for a suitable matrix $A_{(w,\bu)}$. Note that this matrix has
size $(n-p) \times (n-p)$, which is the size of the determinant in
Theorem~\ref{thm:Main}. As $R_\bu$ is orthogonal with determinant $1$,
we have $\det J_{(w,\bu)} \phi=\det A_{(w,\bu)}$ and $||R_\bu w||=||w||$.
This yields the following proposition.

\begin{prop} \label{prop:intA}
The expected number of critical tank-one approximations to a standard
Gaussian tensor in $V$ is
\[
I:=\frac{1}{(2\pi)^{N/2}}
\int_W \int_U
	|\det A_{(w,\bu)}| e^{-\frac{||w||^2}{2}}
	\dd \bu\ \dd w.
\]
\end{prop}

For later use, consider the function $F:U \to \RR$ defined as
\[ F(u)=\frac{1}{(2\pi)^{N/2}}\int_W |\det A_{(w,\bu)}| e^{-\frac{||w||^2}{2}} \dd w. \]
From \eqref{inv2} and the fact that the Gaussian density on $W$ is
orthogonally invariant, it follows that $F$ is invariant under the group
$\prod_{i=1}^p \SO(e_i^\perp)$. In particular, its value depends only
on the tuple $(||u_1||,\ldots,||u_p||)=:(t_1,\ldots,t_p)$. This will be used in the
following subsection.

\subsection{The shape of $A_{(w,\bu)}$}
Recall that $U=\prod_{i=1}^p (e_i^\perp)$. Correspondingly, the columns
of the matrix $A_{(w,\bu)}$ come in $p$ blocks, one for each $e_i^\perp$.
The $i$-th block records the $W^\bot$-components of the vectors
$\left(R_{\bu}^{-1}\frac{\partial R_{\bu}}{\partial
\bv_i}\right)w$, where $\bv_i=(0,\ldots,v_i,\ldots,0)$ and $v_i$ runs
through an orthonormal basis $e_i^{(1)},\ldots,e_i^{(n_i-1)}$ of
$e_i^\perp$. We have
\begin{equation}\label{Ru}
R_{\bu}^{-1}\frac{\partial R_{\bu}}{\partial \bv_i}=\mathrm{Id}\otimes
\cdots \otimes R_{u_i}^{-1}\frac{\partial R_{u_i}}{\partial v_i}\otimes
\cdots \otimes \mathrm{Id}.
\end{equation}
Furthermore, if $v_i$ is also perpendicular to $u_i$, then by \ref{derv} and \ref{inv}
\begin{equation}\label{formperp}
R_{u_i}^{-1}\frac{\partial R_{u_i}}{\partial
v_i}=\frac{1}{\sqrt{1+||
u_i||^2}}(v_ie_i^T-e_iv_i^T)+\frac{1-\sqrt{1+||u_i||^2}}{||u_i||^2\sqrt{1+||u_i||^2}}(v_iu_i^T-u_iv_i^T).
\end{equation}
On the other hand, when $v_i$ is parallel to $u_i$, then
\begin{equation}\label{formpara}
R_{u_i}^{-1}\frac{\partial R_{u_i}}{\partial
v_i}=\frac{1}{1+||u_i||^2}(v_ie_i^T-e_iv_i^T).
\end{equation}
This is derived from \eqref{inv} and \eqref{deru}, keeping in mind the fact
that here $v_i$ needs not be {\em equal} to $u_i$, but merely {\em
parallel} to it. Note that both matrices are skew-symmetric. This is no
coincidence: the directional derivative $\partial R_{u_i}/\partial v_i$
lies in the tangent space to $\SO(V_i)$ at $u_i$, and left multiplying
by $R_{u_i}^{-1}$ maps these elements into the Lie algebra of $\SO(V_i)$,
which consists of skew symmetric matrices.

We decompose the space $W$ as
\begin{align*}
W=&\left(\bigoplus_{i=1}^p e_1 \otimes \cdots \otimes
(e_i)^\bot \otimes \cdots \otimes e_p\right)^\bot
=\mathbb{R}\cdot e_1\otimes e_2\otimes\cdots\otimes
e_p\\ &\oplus \left(\bigoplus_{1\leq i<j\leq p} e_1\otimes\cdots\otimes
e_i^\bot\otimes\cdots\otimes e_j^\bot\otimes\cdots\otimes e_p\right) \oplus W'
=:W_0 \oplus W',
\end{align*}
where $W'$ contains the summands that contain at least three $e_i^\perp$-s
as factors. From \eqref{Ru} it follows that $R_{\bu}^{-1}\frac{\partial
R_{\bu}}{\partial \bv_i}W' \subseteq W$. So for a general $w$
we use the parameters
\[
w=w_0\cdot e_1\otimes\cdots\otimes e_p+\sum_{1\leq i<j\leq p} \sum_{1\leq
a\leq n_i-1} \sum_{1\leq b\leq n_j-1} w_{i,j}^{a,b}
e_1\otimes\cdots\otimes e_i^{(a)}\otimes\cdots\otimes
e_j^{(b)}\otimes\cdots\otimes e_p+ w',
\]
where $w_0$ and $w_{i,j}^{a,b}$ are real numbers, and where $w'
\in W'$ will not contribute to $A_{(w,\bu)}$. We also write
$w_1=(w_0,(w_{i,j}^{a,b}))$ for the components of $w$ that do contribute.

As a further simplification, we take each $u_i$ equal to a
scalar $t_i \geq 0$ times the first basis vector $e_i^{(1)}$
of $e_i^\perp$. This is justified by the observation that the function $F$
is invariant under the group $\prod_i \SO(e_i^\perp)$. Thus we want to determine
$A_{\left(w,(t_1e_1^{(1)},t_2e_2^{(1)},\ldots,t_pe_p^{(1)})\right)}$.
This matrix has a natural block structure $(B_{i,j})_{1\leq
i,j\leq p}$, where $B_{i,j}$ is the part of the Jacobian containing the
$e_1\otimes \cdots \otimes e_i^\bot\otimes\cdots \otimes e_p$-coordinates
of $\left(R_{\bu}^{-1}\frac{\partial R_{\bu}}{\partial
\bv_j}\right)w$ with $\bv_j=(0,\ldots,v_j,\ldots,0)$.

Fixing $i<j$, the matrix $B_{i,j}$ is of type $(n_i-1)\times (n_j-1)$,
where the $(a,b)$-th element is the $e_1\otimes\cdots\otimes
e_i^{(a)}\otimes\cdots\otimes e_p$-coordinate of
\[\left(R_{u_j}^{-1}\frac{\partial R_{u_j}}{\partial e_j^{(b)}}\right)
w.\]
First, if $b\neq1$, then we have a directional derivative
in a direction perpendicular to $u_j=t_j e_j^{(1)}$. Applying formula
\ref{formperp} for the directions $e_j^{(b)}$ yields
\[
B_{i,j}(a,b)= \frac{-w_{i,j}^{a,b}}{\sqrt{1+t_j^2}}.
\]
Second, if $b=1$, then we consider directional derivatives parallel to $u_j$, so applying formula \ref{formpara} for direction $e_j^{(1)}$, we get
\[
B_{i,j}(a,1)=\frac{-w_{i,j}^{a,1}}{1+t_j^2}.
\]
Putting all together, the matrix $B_{i,j}$ is as follows
\[
B_{i,j}=\left(\frac{1}{1+t_j^2}C_{i,j}^1, \frac{1}{\sqrt{1+t_j^2}}
C_{i,j}^2,\ldots,\frac{1}{\sqrt{1+t_j^2}} C_{i,j}^{n_j-1}\right),
\] where $C_{i,j}^b=\left(-w_{i,j}^{a,b}\right)_{1\leq a\leq n_i-1}$ are column vectors for all $1\leq b\leq n_j-1$. Denote the matrix consisting of these column vectors by $C_{i,j}$.
Doing the same calculations but now for the matrix $B_{j,i}$, and
writing $C_{j,i}=C_{i,j}^T$, we find that
\[
B_{j,i}=\left(\frac{1}{1+t_i^2}C_{j,i}^1, \frac{1}{\sqrt{1+t_i^2}}
C_{j,i}^2,\ldots,\frac{1}{\sqrt{1+t_i^2}} C_{j,i}^{n_i-1}\right).
\]
The only remaining case is when $i=j$, and then similar
calculations yield that $B_{j,j}=\frac{1}{(1+t_j^2)^{\frac{n_j}{2}}} w_0 I_{n_j-1}$. We summarize the
content of this subsection as follows.

\begin{prop}\label{form}
For $(w,\bu) \in W \times U$ with $\bu=(t_1e_1^{(1)},\ldots,t_pe_p^{(1)})$
we have
\[
\det A_{(w,\bu)}=\prod_{k=1}^{p}\frac{1}{(1+t_k^2)^{\frac{n_k}{2}}}\det \left(
                                                                \begin{array}{cccc}
                                                                  C_1 &
								  C_{1,2}
								  & \cdots & C_{1,p} \\
                                                                  C_{1,2}^T
								  & C_2 &
								  \cdots& C_{2,p} \\
                                                                  \vdots & \vdots &  & \vdots \\
                                                                  C_{1,p}^T
								  &
								  C_{2,p}^T
								  & \cdots & C_p \\
                                                                \end{array}
                                                              \right),
\]
where $C_{i,j}=\left(-w_{i,j}^{a,b}\right)_{a,b}$ and
$C_j=w_0 I_{n_j-1}$ for all $1\leq i<j\leq p$.
\end{prop}
For further reference we denote the above matrix $(C_{i,j})_{1\leq i,j\leq
p}$ by $C(w_1)$.

\subsection{The value of $I$}
We are now in a position to prove our formula for the expected number of
critical rank-one approximations to a Gaussian tensor $v$.

\begin{proof}[Proof of Theorem~\ref{thm:Main}.]
Combine Propositions~\ref{prop:intA} and~\ref{form} into the
expression
\[
I=\frac{1}{(2\pi)^{\frac{N}{2}}}\prod_{k=1}^p
\mathrm{Vol}(S^{n_k-2})\int\limits_{W}\int\limits_{0}^{\infty}\cdots\int\limits_{0}^{\infty}\prod_{i=1}^p\frac{t_i^{n_i-2}}{(1+t_i^2)^{\frac{n_i}{2}}}\left|\det
C(w_1)\right|e^{-\frac{||w||^2}{2}}\dd t_1\cdots\dd t_p\dd w.
\]
Here the factors $t_i^{n_i-2}$ and the volumes of the sphere account for
the fact that $F$ is orthogonally invariant and $\dd u_i=t_i^{n_i-2}
d_t d_S$, where $d_S$ is the surface element of the $(n_i-2)$-dimensional
unit sphere in $e_i^\perp$. Now recall that
\[
\int\limits_0^{\infty}\frac{t^{n_i-2}}{(1+t^2)^{\frac{n_i}{2}}}\dd
t=\frac{\sqrt{\pi}}{2}\frac{\Gamma(\frac{n_i-1}{2})}{\Gamma(\frac{n_i}{2})},
\] and that the volume of the $(n-2)$-sphere is
\[
\mathrm{Vol}(S^{n_i-2})=\frac{2\pi^{\frac{n_i-1}{2}}}{\Gamma(\frac{n_i-1}{2})}.
\]
Plugging in the above two formulas, we obtain
\[
I=\frac{\sqrt{\pi}^{n}}{\sqrt{2\pi}^{N}}\frac{1}{\prod_{i=1}^p\Gamma\left(\frac{n_i}{2}\right)}\int\limits_W \left|\det C(w)\right|e^{-\frac{||w||^2}{2}}\dd w.
\]
Now the integral splits as an integral over $W_1$ and one over $W'$:
\begin{align*}
&\int\limits_W \left|\det C(w)\right|e^{-\frac{||w||^2}{2}}\dd
=\int\limits_{W'}e^{-\frac{||w'||^2}{2}}\dd w'
\int\limits_{W_1}
\left|\det C(w_1)\right|e^{-\frac{||w_1||^2}{2}}\dd w_1\\
&=\sqrt{2\pi}^{\mathrm{dim}W}\left(\frac{1}{\sqrt{2\pi}^{\mathrm{dim}W_1}}\int\limits_{W_1}
\left|\det C(w_1)\right|e^{-\frac{||w_1||^2}{2}}\dd w_1\right)\\
&=\sqrt{2\pi}^{N-(n-p)} \mathbb{E}(|\det C(w_1)|)
\end{align*}
where $w_1$ is drawn from a standard Gaussian distribution on $W_1$.
Inserting this in the expression for $I$ yields the expression for $I$ in
Theorem~\ref{thm:Main}.
\end{proof}

\subsection{The matrix case}
In this section we perform a sanity check, namely, we show that our
formula in Theorem~\ref{thm:Main} gives the correct answer for the case
$p=2$ and $n_1=n_2=n$---which is $n$, the number of singular values of
any sufficiently general matrix. In this special case we compute
\begin{align*}
J:&=\int\limits_{W}\left|\det C(w)\right| \dd \mu_{W}=\int \limits_{-\infty}^{\infty}\int \limits_{\mathrm{M}_{n-1}} \left|\det\left(
                                                     \begin{array}{cc}
                                                       w_0 I_{n-1} & B \\
                                                       B^T & w_0 I_{n-1} \\
                                                     \end{array}
                                                   \right)\right|e^{\frac{||w_0^2||}{2}}
						   \dd \mu_B \dd w_0=\\
&=\int \limits_{-\infty}^{\infty}\int \limits_{\mathrm{M}_{n-1}}\left|\det(w_0^2 I_{n-1}-BB^T)\right|e^{\frac{||w_0||^2}{2}} \dd \mu_B \dd w_0,
\end{align*}
where $B\in \mathrm{M}_{n-1}(\mathbb{R})$ is a real $(n-1)\times (n-1)$
matrix. The matrix $A:=BB^T$ is a symmetric positive definite matrix and
since the entries of $B$ are independent and normally distributed, $A$
is drawn from the Wishart distribution with density $W(A)$ on the cone
of real symmetric positive definite matrices \cite[Section 2.1]{Rouault07}.
Denote this space by $\Sym_{n-1}$. So the integral we want to calculate is
\[
J=\int \limits_{-\infty}^{\infty}\int \limits_{\Sym_{n-1}}\left|\det(w_0^2 I_{n-1}-A)\right|e^{\frac{||w_0||^2}{2}} \dd W(A) \dd w_0.
\] Now by \cite[Part 2.2.1]{Rouault07} the joint probability density of the eigenvalues $\lambda_j$ of $A$ on the orthant $\lambda_j>0$ is
\begin{equation}\label{dist}
\frac{1}{Z(n-1)}\prod_{j=1}^{n-1}\frac{e^{\frac{-\lambda_j}{2}}}{\sqrt{\lambda_j}}\prod_{1\leq j<k<n-1}|\lambda_k-\lambda_j|,
\end{equation}
where the normalizing constant is
\[
Z(n-1)=\sqrt{2}^{(n-1)^2}\left(\frac{2}{\sqrt{\pi}}\right)^{n-1}\prod_{j=1}^{n-1}\Gamma\left(1+\frac{j}{2}\right)\Gamma\left(\frac{n-j}{2}\right).
\]
Using this fact we obtain
\[
J=\frac{1}{Z(n-1)}\int \limits_{\mathbb{R}}\int \limits_{\lambda >0}\prod_{j=1}^{n-1}\frac{e^{\frac{-\lambda_j}{2}}}{\sqrt{\lambda_j}}\prod_{1\leq j<k<n-1}|\lambda_k-\lambda_j| \prod_{j=1}^{n-1}|w_0^2-\lambda_j|e^{\frac{||w_0||^2}{2}}\dd \lambda \dd w_0.
\]
Now making the change of variables $w_0^2=\lambda_n$, so that
\[
J=2\frac{Z(n)}{Z(n-1)}.
\] Plugging in the remaining normalizing constants we find that the expected number of
critical rank-one approximations to an $n\times n$-matrix is
\[
I=\frac{\sqrt{\pi}^{2n}}{\sqrt{2\pi}^{n^2}}\Gamma\left(\frac{n}{2}\right)^{-2}2\frac{Z(n)}{Z(n-1)}=n.
\]

\section{Symmetric tensors} \label{sec:SymTensors}

\subsection{Set-up}
Now we turn our attention from arbitrary tensors to symmetric tensors,
or, equivalently, homogeneous polynomials. For this, consider $\mathbb{R}^n$
with the standard orthonormal basis $e_1,e_2,\ldots, e_n$ and let
$V=S^p\mathbb{R}^n$ be the space of homogeneous polynomials of degree
$p$ in $n$ variables $e_1,e_2,\ldots, e_n.$  Recall that, up to a
positive scalar, $V$ has a unique inner product that is preserved by
the orthogonal group $\OO_n$ in its natural action on polynomials in
$e_1,\ldots,e_n$. This inner product, sometimes called the {\em Bombieri inner
product}, makes the monomials $e^\sigma:=\prod_i
e_i^{\alpha_i}$ (with $\sigma \in \ZZ_{\geq 0}^n$
and $\sum_i \sigma_i=p$, which we will abbreviate to $\sigma
\vdash p$) into an orthogonal basis with square norms
\[ (e^\sigma|e^\sigma)=\frac{\sigma_1! \cdots
\sigma_n!}{p!}=:\binom{p}{\sigma}^{-1}. \]
The scaling ensures that that the squared norm of a pure power
$(t_1e_1+\ldots+t_n e_n)^p$ equals $(\sum_i t_i^2)^p$. The
scaled monomials
\[f_{\sigma}:=\sqrt{\binom{p}{\sigma}} e^\sigma \]
form an orthonormal basis of $V$, and we equip $V$ with the standard
Gaussian distribution relative to this orthonormal basis.

Now our variety $X$ can be defined by the parameterization
\begin{align*}
\psi:\mathbb{R}^n &\to S^p\mathbb{R}^n,\\
t &\mapsto t^p=\sum_{\sigma\vdash p}
t_1^{\sigma_1} \cdots t_n^{\sigma_n} \sqrt{\binom{p}{\sigma}} f_{\sigma}.
\end{align*}
In fact, if $p$ is odd, then this parameterization is one-to-one, and
$X=\im \psi$. If $p$ is even, then this parameterization is two-to-one,
and $X=\im \psi \cup (-\im \psi)$.

\begin{de}
Define $\Crit$ to be the subset of $V \times X$ consisting of all pairs
of $(v,x)$ such that $v-x \perp T_{x} X$.
\end{de}

\subsection{Parameterizing $\Crit$}
We derive a convenient parameterization of $\Crit$, as follows. Taking
the derivative of $\psi$ at $t \neq 0$, we find that $T_{\pm t^p} X$
both equal $t^{p-1} \cdot \RR^n$. In particular, for $t$ any non-zero
scalar multiple of $e_1$, this tangent space is spanned by all monomials
that contain at least $(n-1)$ factors $e_1$. Let $W$ denote the orthogonal
complement of this space, which is spanned by all monomials that contain
at most $(p-2)$ factors $e_1$. For $u \in e_1^\perp \setminus\{0\}$, recall
from Subsection~\ref{ssec:Parm} the orthogonal map $R_u \in \SO_n$ that
is the identity on $\langle e_1,u \rangle^\perp$ and a rotation sending
$e_1$ to a scalar multiple of $e_1+u$ on $\langle e_1,u \rangle$. We write
$S^p R_u$ for the induced linear map on $V$, which, in particular, sends
$e_1^p$ to $(e_1+u)^p$. We have the following
parameterization of $\Crit$:
\begin{align*}
e_1^\perp \times \RR e_1^p \times W &\to \Crit,\\
(u,w_0 e_1^p,w) &\mapsto (w_0 S^p R_u e_1^p, w_0 S^p R_u e_1^p +
S^p R_u w).
\end{align*}
Combining with the projection to $V$, we obtain the map
\[ \phi:e_1^\perp \times \RR e_1^p \times W \to V, \
(u,w_0 e_1^p,w) \mapsto S^p R_u(w_0 e_1^p + w). \]
Following the strategy in Section~\ref{sec:DoubleCounting}, the expected number
of critical points of $d_v$ on $X$ for a Gaussian $v$ equals
\[ I:=\frac{1}{(2\pi)^{\dim V/2}}
\int_{e_1^\perp} \int_{-\infty}^\infty \int_W
|\det J_{(u,w_0,w)} \phi| e^{-(w_0^2+||w||^2)/2} \dd w \dd
w_0 \dd u, \]
where we have used that $S^p R_u$ preserves the norm, and that $w
\perp e_1^p$.

To determine the Jacobian determinant, we observe that $J_{(u,w_0,w)} \phi$
restricted to $T_{w_0 e_1^p} \RR e_1^p \oplus T_w W$ is just the linear map $S^p
R_u$. Hence, relative to a block decomposition $V=(W+\RR e_1^p)^\perp
\oplus \RR e_1^p \oplus W$ we find
\[
S^p(R_u)^{-1} J_{(u,w_0,w)}\phi=
\left[
\begin{array}{c:c:c}
    A_{(u,w_0,w)} & 0 & 0\\ \hdashline
    \ast & 1& 0\\ \hdashline
    \ast & 0 & I\\
\end{array}
\right]
\]
for a suitable linear map $A_{(u,w_0,w)}:e_1^\perp \to (W \oplus \RR e_1^p)^\perp$.

\subsection{The shape of $A_{(u,w_0,w)}$}

For the computations that follow, we will need only part of our
orthonormal basis of $V$, namely, $e_1^p$ and the vectors
\begin{align*}
f_i&:=\sqrt{p} e_1^{p-1} e_i \\
f_{ii}&:=\sqrt{p(p-1)/2} e_1^{p-2} e_i^2\\
f_{ij}&:=\sqrt{p(p-1)} e_1^{p-2} e_i e_j
\end{align*}
where $2 \leq i \leq n$ in the first two cases and $2 \leq i < j \leq n$
in the last case. The target space of $A_{(u,w_0,w)}$ has an orthonormal
basis $f_2,\ldots,f_n$, while the domain has an orthonormal basis
$e_2,\ldots,e_n$. Let $a_{kl}$ be the coefficient of
$f_k$ in $A_{(u,w_0,w)} e_l$. To compute $a_{kl}$, we expand $w$ as
\[
w=
\sum_{2 \leq i \leq j} w_{ij} f_{ij} + w' =: w_1 + w'
\]
where $w'$ contains the terms with at most $p-3$ factors
$e_1$. We have the identity
\[ S^p (R_u)^{-1} \frac{\partial S^p R_u (e_{i_1} \cdots
e_{i_p})}{\partial e_l} =
\sum_{m=1}^p e_{i_1} \cdots (R_u^{-1} \frac{\partial
R_u}{\partial e_l} e_{i_m}) \cdots  e_{i_p}. \]
For this expression to contain terms that are multiples of some $f_k$,
we need that at least $p-2$ of the $i_m$ are equal to $1$. Thus $a_{kl}$
is independent of $w'$, which is why we need only the basis vectors above.

As in the case of ordinary tensors, we make the further simplification
that $u=t e_2$. Then we have to distinguish two cases: $l=2$ and $l>2$.
For $l=2$ formula \eqref{formpara} applies, and we compute modulo $\langle
f_2,\ldots,f_n \rangle^\perp$
\begin{align*}
&(S^p R_{te_2})^{-1} \frac{\partial (S^p R_{te_2} (w_0 e_1^p +
w_1))}{\partial e_2}\\ &=
(S^p R_{te_2})^{-1} \frac{\partial (S^p R_{te_2} (w_0 e_1^p +
\sum_{2 \leq i} w_{ii} f_{ii} +
\sum_{2 \leq i<j} w_{ij} f_{ij}))}{\partial e_2}\\
&=
\frac{1}{1+t^2}(
p w_0 e_1^{p-1} e_2
- 2 w_{22} \sqrt{p(p-1)/2} e_1^{p-1} e_2
- \sum_{2<j} w_{2j} \sqrt{p(p-1)} e_1^{p-1} e_j)\\
&=
\frac{1}{1+t^2}(
(\sqrt{p} w_0 - \sqrt{2(p-1)} w_{22}) f_{2}
- \sum_{2<j} \sqrt{p-1} w_{2j} f_j ).
\end{align*}
For $l>2$ formula \eqref{formperp} applies, but in fact the second
term never contributes when we compute modulo $\langle f_2,\ldots,f_n
\rangle^\perp$:
\begin{align*}
&(S^p R_{te_2})^{-1} \frac{\partial (S^p R_{te_2} (w_0 e_1^p +
w_1))}{\partial e_l}\\
&=
(S^p R_{te_2})^{-1} \frac{\partial (S^p R_{te_2} (w_0 e_1^p +
\sum_{2 \leq i} w_{ii} f_{ii} +
\sum_{2 \leq i<j} w_{ij} f_{ij}))}{\partial e_l}\\
&=
\frac{1}{\sqrt{1+t^2}}
\left(p w_0 e_1^{p-1} e_l
- 2 w_{ll} \sqrt{p(p-1)/2} e_1^{p-1} e_l \right.\\
&- \left. \sqrt{p(p-1)}(\sum_{2 \leq i<l} w_{il} e_1^{p-1} e_i
+ \sum_{l < j} w_{2j} e_1^{p-1} e_j)\right)\\
&=\frac{1}{\sqrt{1+t^2}}\left( (\sqrt{p}w_0 -
\sqrt{2(p-1)}w_{ll}) f_l - \sum_{i \neq l} \sqrt{p-1} w_{il}
f_i \right);
\end{align*}
here we use the convention that $w_{il}=w_{li}$ if $i>l$. We
have thus proved the following proposition.

\begin{prop}
The determinant of $A_{(te_2,w_0,w)}$ equals
\[
\frac{1}{(1+t^2)^{n/2}} \det \left(\sqrt{p}w_0 I
- \sqrt{p-1}
\cdot
\begin{bmatrix}
\sqrt{2} w_{22} & w_{23} & \cdots & w_{2n}\\
w_{23} & \sqrt{2} w_{33} & \cdots & w_{3n}\\
\vdots & \vdots & & \vdots \\
w_{2n} & w_{3n} & \cdots & \sqrt{2} w_{nn}
\end{bmatrix}\right).
\]
\end{prop}

We denote the $(n-1) \times (n-1)$-matrix by $C(w_1)$.

\subsection{The value of $I$}

We can now formulate our theorem for symmetric tensors.

\begin{prop} \label{prop:Symm}
For a standard Gaussian random symmetric tensor $v \in S^p \RR^n$
(relative to the Bombieri norm) the expected number of critical points
of $d_v$ on the manifold of non-zero symmetric tensors of rank one equals
\[
\frac{\sqrt{\pi}}{2^{(n-1)/2} \Gamma(\frac{n}{2})}
\mathbb{E}(|\det(\sqrt{p}w_0 I - \sqrt{p-1}C(w_1))|),
\]
where $w_0$ and the entries of $w_1$ are independent and
$\sim\mathcal{N}(0,1)$.
\end{prop}

\begin{proof}
Combining the results from the previous subsections, we find
\begin{align*}
&I=\frac{1}{(2\pi)^{\dim V/2}} \mathrm{Vol}(S^{n-2})\\
&\cdot \int_0^\infty \int_{-\infty}^\infty \int_W
|\det(\sqrt{p}w_0 I - \sqrt{p-1}C(w_1))|
e^{-\frac{w_0^2+||w||^2}{2}} \frac{t^{n-2}}{(1+t^2)^{n/2}} \dd w \dd w_0 \dd t.
\end{align*}
Here, like in the ordinary tensor case, we have used that the function
$F(u)$ in the definition of $I$ is $\OO(e_1^\perp)$-invariant.
Now plug in
\[
\int\limits_0^{\infty}\frac{t^{n-2}}{(1+t^2)^{\frac{n}{2}}}\dd
t=\frac{\sqrt{\pi}}{2}\frac{\Gamma(\frac{n-1}{2})}{\Gamma(\frac{n}{2})}
\text{ and }
\mathrm{Vol}(S^{n-2})=\frac{2\pi^{\frac{n-1}{2}}}{\Gamma(\frac{n-1}{2})}
\]
to find that $I$ equals
\[
\frac{1}{2^{\dim V/2} \pi^{(\dim V-n)/2}
\Gamma(\frac{n}{2})}
\cdot \int_{-\infty}^\infty \int_W
|\det(\sqrt{p}w_0 I - \sqrt{p-1}C(w_1))|
e^{-\frac{w_0^2+||w||^2}{2}} \dd w \dd w_0.
\]
Finally, we can factor out the part of the integral concerning $w'$, which
lives in a space of dimension $\dim V - 1 - (n-1) - n(n-1)/2 = \dim V -
n(n+1)/2$. As a consequence, we need only integrate over the space $W_1$
where $w_1$ lives, and have to multiply by a suitable power
of $2\pi$:
\begin{align*}
I=
&\frac{1}{2^{n(n+1)/4} \pi^{n(n-1)/4}
\Gamma(\frac{n}{2})}\\
&\cdot \int_{-\infty}^\infty \int_{W_1}
|\det(\sqrt{p}w_0 I - \sqrt{p-1}C(w_1))|
e^{-\frac{w_0^2+||w_1||^2}{2}} \dd w_1 \dd w_0\\
&=\frac{\sqrt{\pi}}{2^{(n-1)/2} \Gamma(\frac{n}{2})}
\mathbb{E}(|\det(\sqrt{p}w_0 I - \sqrt{p-1}C(w_1))|)
\end{align*}
as desired.
\end{proof}

\subsection{Further dimension reduction}

Since the matrix $C$ from Proposition~\ref{prop:Symm} is just $\sqrt{2}$
times a random matrix from the standard Gaussian orthogonal ensemble,
and in particular has an orthogonally invariant probability density,
we can further reduce the dimension of the integral, as follows.

\begin{proof}[Proof of Theorem~\ref{thm:Symm}.]
First we denote the diagonal entries of $C$
\[\tilde{w}_{ii}:=\sqrt{2}w_{ii},\ i=2,\ldots,n\]
Then the joint density function of the random matrix $C$ equals
\[ f_{n-1}(\tilde{w}_{ii},w_{ij}):=\frac{1}{2^{(n-1)/2} \cdot
(2\pi)^{n(n-1)/4} }
e^{-(\tilde{w}_{22}^2+\cdots+\tilde{w}_{nn}^2)/4 - \sum_{2
\leq i<j \leq n} w_{ij}^2/2}. \]
This function is invariant under conjugating $C$ with an
orthogonal matrix, and as a consequence, the joint density of the ordered
tuple $(\lambda_2 \leq \ldots \leq \lambda_n)$ of
eigenvalues of $C$ equals
\[
Z(n-1)f_{n-1}(\Lambda)\prod_{i< j}(\lambda_j-\lambda_i),
\]
(see \cite[Theorem 3.2.17]{Muirhead82}\footnote{The theorem
there concerns the positive-definite case, but is true for
orthogonally invariant density functions on general
symmetric matrices.}). Here $\Lambda$ is the diagonal
matrix with $\lambda_2,\ldots,\lambda_n$ on the diagonal, and
\[Z(n-1)=\frac{\pi^{n(n-1)/4}}{\prod_{i=1}^{n-1}\Gamma(i/2)}.\]
Consequently, we have
\begin{align*}
I&=\frac{\sqrt{\pi}}{2^{(n-1)/2}\Gamma(\frac{n}{2})}
\int\limits_{\lambda_2\leq \ldots\leq
\lambda_{n}}\int\limits_{-\infty}^{\infty}\left(\prod_{i=2}^{n}|\sqrt{p}w_0-
\sqrt{p-1} \lambda_i|\right)
\left(\prod_{i<j}(\lambda_j-\lambda_i)
\right)\\
&\cdot Z(n-1)f_{n-1}(\Lambda)\left(\frac{1}{\sqrt{2\pi}}e^{-w_0^2/2}\right)\dd
w_0 \dd \lambda_2 \cdots \dd \lambda_n.\\
&=\frac{1}{2^{(n^2+3n-2)/4} \prod_{i=1}^n
\Gamma(i/2)}
\int\limits_{\lambda_2\leq \ldots\leq
\lambda_{n}}\int\limits_{-\infty}^{\infty}\left(\prod_{i=2}^{n}|\sqrt{p}w_0-
\sqrt{p-1} \lambda_i|\right)\\
&\cdot
\left(\prod_{i<j}(\lambda_j-\lambda_i)
\right) e^{-w_0^2/2-\sum_{i=2}^n \lambda_i^2/4}
\dd
w_0 \dd \lambda_2 \cdots \dd \lambda_n,
\end{align*}
as required.
\end{proof}

\subsection{The cone over the rational normal curve}
In the case where $n=2$, the integral from Theorem~\ref{thm:Symm} is
over a $2$-dimensional space and can be computed in closed form.

\begin{thm}
For $n=2$ the number of critical points in
Theorem~\ref{thm:Symm} equals $\sqrt{3p-2}$.
\end{thm}

A slightly different computation yielding this result can
be found in \cite{Draisma13c}.

\subsection{Veronese embeddings of the projective plane}
In the case where $n=3$, the integral from Theorem~\ref{thm:Symm} gives
the number of critical points to the cone over the $p$-th Veronese
embedding of the projective plane. In this case the integral can be
computed in closed form, using symbolic integration in {\tt Mathematica}
we have the following result.

\begin{thm}
For $n=3$ the number of critical points in
Theorem~\ref{thm:Symm} equals
\[
1+4\cdot\frac{p-1}{3p-2}\sqrt{(3p-2)\cdot(p-1)}.
\]
\end{thm}

We do not know whether a similar closed formula exists for higher values
of $n$.

\subsection{Symmetric matrices} \label{ssec:Symmats}
In Example~\ref{ex:Symm} we saw that the case where $p=2$ concerns
rank-one approximations to symmetric matrices, and that the average
number of critical points is $n$.  We now show that the integral above
also yields $n$. Here we have
\begin{align*}
I&=\frac{\sqrt{\pi}}{2^{(n-1)/2}\Gamma(\frac{n}{2})}
\int\limits_{\lambda_2\leq \ldots\leq
\lambda_{n}}\int\limits_{-\infty}^{\infty}\left(\prod_{i=2}^{n}|\sqrt{2}w_0-
\lambda_i|\right)
\left(\prod_{i<j}(\lambda_j-\lambda_i)
\right)\\
&\cdot Z(n-1)f_{n-1}(\Lambda)\left(\frac{1}{\sqrt{2\pi}}e^{-w_0^2/2}\right)\dd
w_0 \dd \lambda_2 \cdots \dd \lambda_n.
\end{align*}
Now set $\lambda_1:=\sqrt{2}w_0$. Then the inner integral over $\lambda_1$
splits into $n$ integrals, according to the relative position of
$\lambda_1$ among $\lambda_2 \leq \cdots \leq \lambda_n$.
Moreover, these integrals are all equal. Hence we find
\begin{align*}
I&=n\frac{\sqrt{\pi}}{2^{(n-1)/2}\Gamma(\frac{n}{2})}
\int\limits_{\lambda_1\leq \ldots\leq
\lambda_{n}}
\left(\prod_{1 \leq i<j \leq n}(\lambda_j-\lambda_i)
\right)\\
&\cdot Z(n-1)
\cdot
\frac{1}{2^{n/2} \cdot (2\pi)^{(n(n-1)+2)/4}}
e^{-(\lambda_{1}^2+\cdots+\lambda_n^2)/4}
\dd \lambda_1 \cdots \dd \lambda_n\\
&=n\frac{\sqrt{\pi}}{2^{(n-1)/2}\Gamma(\frac{n}{2})}
\int\limits_{\lambda_1\leq \ldots\leq
\lambda_{n}}
\left(\prod_{1 \leq i<j \leq n}(\lambda_j-\lambda_i)
\right)\\
&\cdot Z(n-1)
\cdot
f_n(\diag(\lambda_1,\ldots,\lambda_n)) \cdot (2\pi)^{(n-1)/2}
\dd \lambda_1 \cdots \dd \lambda_n.
\end{align*}
Now, again by \cite[Theorem 3.2.17]{Muirhead82}, the integral of $\prod_{1 \leq
i<j \leq n} (\lambda_j-\lambda_i) \cdot f_n$ equals $1/Z(n)$. Inserting
this into the formula yields $I=n$.

\section{Values} \label{sec:Values}

In this section we record some values of the expressions in
Theorem~\ref{thm:Main} and~\ref{thm:Symm}.

\subsection{Ordinary tensors}

Below is a table of expected numbers of critical rank-one approximations
to a Gaussian tensor, computed from Theorem~\ref{thm:Main}.  We also
include the count over $\CC$ from \cite{Friedland12b}. Unfortunately, the
dimensions of the integrals from Theorem~\ref{thm:Main} seem to prevent
accurate computation numerically, at least with all-purpose software
such as {\tt Mathematica}. Instead, we have estimated these integrals as
follows: for some initial value $I$ (we took $I=15$), take $2^I$ samples
of $C$ from the multivariate standard normal distribution, and compute
the average absolute determinant. Repeat with a new sample of size $2^I$,
and compare the absolute difference of the two averages divided by the
first estimate. If this relative difference is $<10^{-4}$, then stop. If not,
then group the current $2^{I+1}$ samples together, sample another $2^{I+1}$,
and perform the same test. Repeat this process, doubling the sample size
in each step, until the relative difference is below $10^{-4}$. Finally,
multiply the last average by the constant in front of the integral in
Theorem~\ref{thm:Main}. We have not computed a confidence interval for
the estimate thus computed, but repetitions of this procedure suggest that
the first three computed digits are correct; we give one more digit below.
\[
\centering
    \begin{tabular}{l|l|l}
      Tensor format & average count over $\RR$ & count over
      $\CC$ \\ \hline
    $n\times m$ & $\min(n,m)$  & $\min(n,m)$ \\
    $2^3=2\times2\times2$         & 4.287         & 6        \\
    $2^4$         & 11.06	& 24        \\
    $2^5$         & 31.56	& 120        \\
    $2^6$         & 98.82	& 720        \\
    $2^7$         & 333.9	& 5040        \\
    $2^8$         & $1.206\cdot 10^3$         & 40320        \\
    $2^9$         & $4.611\cdot 10^3$         & 362880        \\
    $2^{10}$	  & $1.843\cdot 10^4$         & 3628800	\\
    $2\times 2 \times 3$         & 5.604         & 8        \\
    $2\times 2 \times 4$         & 5.556         & 8        \\
    $2\times 2 \times 5$         & 5.536         & 8        \\
    $2\times 3 \times 3$         & 8.817         & 15        \\
    $2\times 3 \times 4$         & 10.39         & 18        \\
    $2\times 3 \times 5$         & 10.28         & 18        \\
    $3\times 3\times 3$          & 16.03         & 37        \\
    $3\times 3\times 4$ 	 & 21.28         & 55 \\
    $3\times 3\times 5$         & 23.13         & 61        \\
    \end{tabular}
\]

Except in some small cases, we do not expect that there exists a closed
form expression for $\mathbb{E}(|\det(C)|)$. However, asymptotic results
on expected absolute determinants such as those in \cite{Tao12} should
give asymptotic results for the counts in Theorems~\ref{thm:Main}
and~\ref{thm:Symm}, and it would be interesting to compare these with
the count over $\CC$.

From \cite{Friedland12b} we know that the count for ordinary tensors
stabilizes for $n_p - 1 \geq \sum_{i=1}^{p-1} (n_i-1)$, i.e., beyond the
{\em boundary format} \cite[Chapter 14]{Gelfand94}, where the variety
dual to the variety of rank-one tensors ceases to be a hypersurface.
We observe a similar behavior experimentally for the average count
according to Theorem~\ref{thm:Main}, although the count seems to {\em
decrease} slightly rather than to {\em stabilize}. It would be nice to
prove this behavior from our formula, but even better to give a geometric
explanation both over $\RR$ and over $\CC$.

\subsection{Symmetric tensors}

The following table contains the average number of rank-one tensor
approximations to $S^p\mathbb{R}^n$ according to Theorem~\ref{thm:Symm}
(on the left). The integrals here are over a much lower-dimensional
domain than in the previous section, and they {\em can} be evaluated
accurately with {\tt Mathematica}. On the right we list the corresponding count over $\CC$. By \cite[Theorem 12]{Friedland12b} these values are simply $1+(p-1)+\cdots+(p-1)^{n-1}$.
\\
    \begin{tabular}{c|cccc}
     $p\backslash n$ & 1 & 2 & 3& 4\\ \hline
    1 & 1  & 1& 1& 1\\
    2 & 1  & 2& 3& 4\\
    3 & 1  & $\sqrt{7}$&
$1+4 \cdot \frac{2}{7} \cdot \sqrt{7\cdot 2}$&
9.3951\\
    4 & 1  & $\sqrt{10}$&
$1+4 \cdot \frac{3}{10} \cdot \sqrt{10\cdot 3}$&
16.254\\
    5 & 1  & $\sqrt{13}$&
$1+4 \cdot \frac{4}{13} \cdot \sqrt{13\cdot 4}$&
24.300\\
    6 & 1  & $\sqrt{16}$&
$1+4 \cdot \frac{5}{16} \cdot \sqrt{16\cdot 5}$&
33.374\\
    7 & 1  & $\sqrt{19}$&
$1+4 \cdot \frac{6}{19} \cdot \sqrt{19\cdot 6}$&
43.370\\
    8 & 1  & $\sqrt{22}$&
$1+4 \cdot \frac{7}{22} \cdot \sqrt{22\cdot 7}$&
54.211\\
    9 & 1  & $\sqrt{25}$&
$1+4 \cdot \frac{8}{25} \cdot \sqrt{25\cdot 8}$&
65.832\\
    10& 1  & $\sqrt{28}$&
$1+4 \cdot \frac{9}{28} \cdot \sqrt{28\cdot 9}$&
78.185\\
    \end{tabular}
    \hfill
\begin{tabular}{c|cccc}
$p\backslash n$ & 1 & 2 & 3& 4\\ \hline
1 & 1 & 1 & 1 & 1\\
2&1& 2& 3& 4\\
3&1& 3& 7& 15\\
4&1& 4& 13& 40\\
5&1& 5& 21& 85\\
6&1& 6& 31& 156\\
7&1& 7& 43& 259\\
8&1& 8& 57& 400\\
9&1& 9& 73& 585\\
10&1& 10& 91& 820
\end{tabular}


\begin{thebibliography}{IAvHdL11}

\bibitem[BCMT10]{Brachat10}
Jerome Brachat, Pierre Comon, Bernard Mourrain, and Elias Tsigaridas.
\newblock Symmetric tensor decomposition.
\newblock {\em Linear Algebra Appl.}, 433(11--12):1851--1872, 2010.

\bibitem[BW08]{Belzen08}
Femke~van Belzen and Siep Weiland.
\newblock Diagonalization and low-rank appromixation of tensors: a singular
  value decomposition approach.
\newblock In {\em Proceedings 18th International Symposium on Mathematical
  Theory of Networks \& Systems (MTNS), 28 July - 1 August 2008, Blacksburg,
  Virginia}, Blacksburg, Virginia, USA, 2008. MTNS.

\bibitem[BW09]{Belzen09}
Femke~van Belzen and Siep Weiland.
\newblock Approximation of n{D} systems using tensor decompositions.
\newblock In {\em Proceedings of the International Workshop on Multidimensional
  (nD) Systems, June 29th - July 1st, 2009, Thessaloniki, Greece}, pages 1--8,
  Piscataway, 2009. IEEE Service Center.

\bibitem[BWG07]{Belzen07}
F.~van Belzen, S.~Weiland, and J.~de Graaf.
\newblock Singular value decompositions and low rank approximations of
  multi-linear functionals.
\newblock In {\em Proceedings of the 46th Conference on Decision and Control
  (CDC 2007) 12-14 December 2007, New Orleans, Louisiana, USA. (pp.
  3751-3756)}, Piscataway, New Jersey, USA, 2007. IEEE.

\bibitem[CGLM08]{Comon08}
Pierre Comon, Gene Golub, Lek-Heng Lim, and Bernard Mourrain.
\newblock Symmetric tensors and symmetric tensor rank.
\newblock {\em SIAM J. Matrix Anal. Appl.}, 30(3):254--1279, 2008.

\bibitem[CS13]{Cartwright13}
Dustin {Cartwright} and Bernd {Sturmfels}.
\newblock {The number of eigenvalues of a tensor.}
\newblock {\em {Linear Algebra Appl.}}, 438(2):942--952, 2013.

\bibitem[DDV00]{DeLathauwer00}
Lieven {De Lathauwer}, Bart {De Moor}, and Joos {Vandewalle}.
\newblock {On the best rank-1 and rank-$(R_1,R_2,. . .,R_N)$ approximation of
  higher-order tensors.}
\newblock {\em {SIAM J. Matrix Anal. Appl.}}, 21(4):1324--1342, 2000.

\bibitem[DHO{\etalchar{+}}16]{Draisma13c}
Jan Draisma, Emil Horobet, Giorgio Ottaviani, Bernd Sturmfels, and Rekha~R.
  Thomas.
\newblock The {E}uclidean distance degree of an algebraic variety.
\newblock {\em Found. Comput. Math.}, 2016.
\newblock To appear; \verb+arXiv:1309.0049+.

\bibitem[DL08a]{DeLathauwer08a}
Lieven De~Lathauwer.
\newblock Decompositions of a higher-order tensor in block terms. {I}: {L}emmas
  for partitioned matrices.
\newblock {\em SIAM J. Matrix Anal. Appl.}, 30(3):1022--1032, 2008.

\bibitem[DL08b]{DeLathauwer08b}
Lieven De~Lathauwer.
\newblock Decompositions of a higher-order tensor in block terms. {I}{I}:
  {D}efinitions and uniqueness.
\newblock {\em SIAM J. Matrix Anal. Appl.}, 30(3):1033--1066, 2008.

\bibitem[DLN08]{DeLathauwer08c}
Lieven De~Lathauwer and Dimitri Nion.
\newblock Decompositions of a higher-order tensor in block terms. {I}{I}{I}:
  {A}lternating least squares algorithms.
\newblock {\em SIAM J. Matrix Anal. Appl.}, 30(3):1067--1083, 2008.

\bibitem[dSL08]{deSilva08}
Vin de~Silva and Lek-Heng Lim.
\newblock Tensor rank and the ill-posedness of the best low-rank approximation
  problem.
\newblock {\em SIAM J. Matrix Anal. Appl.}, 30(3):1084--1127, 2008.

\bibitem[FO12]{Friedland12b}
Shmuel Friedland and Giorgio Ottaviani.
\newblock The number of singular vector tuples and uniqueness of best rank one
  approximation of tensors.
\newblock {\em Found.~Comp.~Math.}, 2012.
\newblock To appear, \verb+arXiv:1210.8316+.

\bibitem[GKZ94]{Gelfand94}
Israel~M. Gelfand, Mikhail~M. Kapranov, and Andrei~V. Zelevinsky.
\newblock {\em Discriminants, resultants, and multidimensional determinants}.
\newblock Mathematics: Theory \& Applications. Birkh\"auser, Boston, MA, 1994.

\bibitem[H{\aa}s90]{Hastad90}
Johan H{\aa}stad.
\newblock Tensor rank is {N}{P}-complete.
\newblock {\em J. Algorithms}, 11(4):644--654, 1990.

\bibitem[HL13]{Hillar10}
Christopher~J. {Hillar} and Lek-Heng {Lim}.
\newblock {Most tensor problems are NP-hard.}
\newblock {\em {J. ACM}}, 60(6):39, 2013.

\bibitem[Hor15]{Horobet15a}
Emil Horobe\c{t}.
\newblock The data singular and the data isotropic loci for affine cones.
\newblock 2015.
\newblock Preprint; \verb+arxiv:1507.02923+.

\bibitem[IAvHdL11]{Ishteva11}
Mariya Ishteva, P.-A. Absil, Sabine van Huffel, and Lieven de~Lathauwer.
\newblock Best low multilinear rank approximation of higher-order tensors,
  based on the {R}iemannian trust-region scheme.
\newblock {\em SIAM J. Matrix Anal. Appl.}, 32(1):115--135, 2011.

\bibitem[Lim05]{Lim05}
Lek-Heng Lim.
\newblock Singular values and eigenvalues of tensors: a variational approach.
\newblock In {\em Proceedings of the IEEE International Workshop on
  Computational Advances in Multi-Sensor Adaptive Processing (CAMSAP '05)},
  volume~1, pages 129--132, 2005.

\bibitem[{Mui}82]{Muirhead82}
Robb~J. {Muirhead}.
\newblock {Aspects of multivariate statistical theory.}
\newblock {Wiley Series in Probability and Mathematical Statistics. New York:
  John Wiley \& Sons, Inc. XIX, 673 p. (1982).}, 1982.

\bibitem[OO13]{Oeding13}
Luke {Oeding} and Giorgio {Ottaviani}.
\newblock {Eigenvectors of tensors and algorithms for Waring decomposition.}
\newblock {\em {J. Symb. Comput.}}, 54:9--35, 2013.

\bibitem[{Rou}07]{Rouault07}
Alain {Rouault}.
\newblock {Asymptotic behavior of random determinants in the Laguerre, Gram and
  Jacobi ensembles.}
\newblock {\em {ALEA, Lat. Am. J. Probab. Math. Stat.}}, 3:181--230, 2007.

\bibitem[TV12]{Tao12}
Terence Tao and Van Vu.
\newblock A central limit theorem for the determinant of a {W}igner matrix.
\newblock {\em Adv. Math.}, 231(1):74--101, 2012.

\end{thebibliography}

\newcommand{\etalchar}[1]{$^{#1}$}

\end{document}